\documentclass{amsart}
\usepackage[utf8x]{inputenc}
\usepackage{amsmath,amsthm}
\usepackage{amsfonts,amstext,amssymb, mathtools,  mathtools,comment}
\RequirePackage[colorlinks,citecolor=blue,urlcolor=blue]{hyperref}

\newcommand{\levy}{L\'{e}vy }
\newcommand{\p}{{\mathbb P}}
\newcommand{\e}{{\mathbb E}}
\newcommand{\D}{{\mathrm d}}

\newcommand{\R}{{\mathbb R}}
\newcommand{\1}[1]{\mbox{\rm\large  1}_{\{#1\}}}

\newcommand{\eqd}{\stackrel{d}{=}}

\newtheorem{cor}{Corollary}[section]

\newtheorem{prop}{Proposition}[section]

\begin{document}
\title[Sparre-Andersen identity]{Sparre-Andersen identity and the last passage time}
\author[J.\ Ivanovs]{Jevgenijs Ivanovs}
\address{Department of Actuarial Science, Faculty of Business and Economics, University of Lausanne, CH-1015 Lausanne, Switzerland}
\begin{abstract}
It is shown that the celebrated result of Sparre Andersen for random walks and \levy processes has intriguing consequences when the last time of the process in $(-\infty,0]$, say $\sigma$, is added to the picture. In the case of no positive jumps this leads to six random times, all of which have the same distribution - the uniform distribution on $[0,\sigma]$. Surprisingly, this result does not appear in the literature, even though it is based on some classical observations concerning exchangeable increments.
\end{abstract}
\maketitle

\section{The main observation}\label{sec:intro}
The main observation of this note is best illustrated by a \levy process $X_t,t\geq 0$ without positive jumps.
A particular example of such $X$ is given by a compound Poisson process with positive linear drift and negative jumps,
which occupies a central place in applied probability: in risk theory it is known as a Cram\'er-Lundberg model~\cite{asmussen_ruin},
and in queueing theory $-X$ drives the workload process in the classical M/G/1 queue~\cite{asmussen_applied}.

\begin{figure}[h!]
  \centering
  \includegraphics{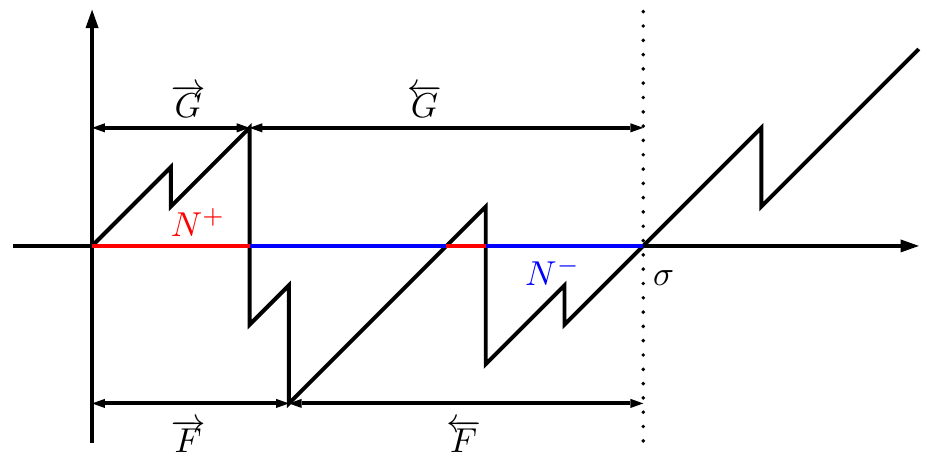}
  \caption{A sample path of $X$ and the corresponding random times.}
\label{fig_CL}
\end{figure}
Let $\sigma=\sup\{t\geq 0:X_t\leq 0\}$ be the last time of $X$ in $(-\infty,0]$, which is finite a.s. when $\e X_1>0$.
Define the following random times, see Fig.~\ref{fig_CL}:
\begin{align*}&N^-=\int_0^\sigma \1{X_s\leq 0}\D s, &N^+=\int_0^\sigma \1{X_s\geq 0}\D s,\\
&\overrightarrow F=\sup\{t\in[0,\sigma):X_t=\underline X_t\}, &\overrightarrow G=\sup\{t\in[0,\sigma):X_t=\overline X_t\},\\
&\overleftarrow F=\sigma-\overrightarrow F, &\overleftarrow G=\sigma-\overrightarrow G,
\end{align*}
where $\overline X_t=\sup\{X_s:s\in[0,t]\}$ and $\underline X_t=\inf\{X_s:s\in[0,t]\}$ are the running supremum and infimum processes respectively.
When $\sigma=0$ we assume that all these times are~0. In words, $N^-$ is the time spent in the non-positive half-line, $\overrightarrow F$ is the time of the infimum, and
$\overleftarrow F$ is the time from the infimum to~$\sigma$.

\begin{prop}\label{prop_nojumps}
Let $X$ be a \levy process without positive jumps, such that $\e X_1>0$. Then $\overrightarrow F, \overleftarrow F, \overrightarrow G, \overleftarrow G, N^-,N^+$ have the same distribution.
\end{prop}
Note that we can replace $\sigma$ by $\infty$ in the definitions of $N^-$ and $\overrightarrow F$.
The equivalence of laws of these two random variables is known as Sparre-Andersen identity, see e.g.~\cite[Lem.\ VI.15]{bertoin}.
This identity for random walks was first established by E.\ Sparre Andersen in~\cite{sparre_andersen1953sums} using combinatorial approach;
a simpler proof can be found in~\cite[Thm.\ XII.8.2]{feller}.

The transform of the single distribution in Proposition~\ref{prop_nojumps} is well known.
Define the first passage time $\tau_x=\inf\{t\geq 0:X_t>x\}$ and let $\psi(s)=\log(\e e^{s X_1}),\Phi(s)=-\log(\e e^{-s \tau_1})$ for $s\geq 0$,
which are known to satisfy $\psi(\Phi(s))=s$.
Then it follows from~\cite[Thm.\ VII.4(ii)]{bertoin} that
\begin{equation}\label{eq:transform}
\e e^{-s \overrightarrow F}=\psi'(0)\frac{\Phi(s)}{s},\quad s> 0.
\end{equation}
Alternatively, $\overleftarrow F$ is the last passage time of the post-infimum process (known as $X$ conditioned to stay positive) over $I=-\underline X_\infty$.
It is well known that the post-infimum process is independent of the infimum and
by William's representation~\cite[Thm.\ VII.18]{bertoin} its last passage time over~$x$ has the law of~$\tau_x$.
Hence we can add the following identity to Proposition~\ref{prop_nojumps}:
\begin{equation}\label{eq:ack}
\overleftarrow F\eqd \hat \tau_{I},
\end{equation}
where $\hat \tau$ is a copy of $\tau$ independent of $X$. In particular, this readily implies that the transform of $\overleftarrow F$ coincides with~\eqref{eq:transform} by way of the generalized Pollaczek-Khinchine formula:
$\e e^{-s I}=\psi'(0)s/\psi(s)$.

Similarly to the classical identity, Proposition~\ref{prop_nojumps} can be reformulated for a \levy process on a finite interval $[0,T]$, see Proposition~\ref{prop_nojumps2} below. Yet another possibility is to consider a general \levy process and to condition on the event $\{X_\sigma=0\}$, assuming it has positive probability.
Corollary~\ref{cor:rw} presents this type of result for random walks. 
Note that if local extrema are not necessarily distinct then $\overleftarrow F$ and $\overleftarrow G$ must be defined in a slightly different way, see Section~\ref{sec:rw}.
\begin{prop}\label{prop_nojumps2}
Let $X$ be a \levy process without positive jumps, such that $\p(X_T> 0)>0$, and let $\sigma=\sup\{t\in[0,T]:X_t\leq 0\}$.
On the event $\{X_T> 0\}$ the random times
$\overrightarrow F, \overleftarrow F, \overrightarrow G, \overleftarrow G, N^-,N^+$ have the same distribution.
\end{prop}

In general, when jumps of both signs are allowed, the above equivalence of laws does not hold.
Instead, we can partition these times into two classes of three elements in each according to their laws.
We state this general result for random walks and provide its short proof in Section~\ref{sec:rw}.
Its standard extension to \levy processes is discussed in short in Section~\ref{sec:levy},
where we also give some additional comments.

\section{Intuitive explanation and further consequences}
There is a simple explanation of the above results: the Sparre-Andersen identity  holds for the random time interval $[0,\sigma]$ (applied to $-X$), and the process seen from~$\sigma$ (backwards in time and downwards in space) has the same law as the original process up to~$\sigma$. The fundamental reason behind these observations is that the increments of the approximating random walk are exchangeable random variables conditioned on~$\{\sigma=n\}$, see Section~\ref{sec:rw} and Section~\ref{sec:levy} for details. 

By considering the approximating random walk we observe some interesting further consequences. Firstly, we notice that the above 6 random times have the same distribution conditional on $\sigma$, and so 
\begin{align}
\text{the pairs }
(\overrightarrow G,\overleftarrow G),(\overrightarrow F,\overleftarrow F),
(N^+,N^-)\text{ have the same distribution}
\end{align}
with exchangeable components.
The corresponding transform (under assumptions of Proposition~\ref{prop_nojumps}) can be obtained in a similar way as above:
\begin{equation}\e e^{-s \overrightarrow F-t \overleftarrow F}=
\e e^{-s \overrightarrow F-t\hat \tau_I}=\e e^{-s \overrightarrow F+\Phi(t)\underline X_\infty}=\psi'(0)\frac{\Phi(s)-\Phi(t)}{s-t}
\end{equation}
using~\eqref{eq:ack} and the explicit form of the Wiener-Hopf factor corresponding to the infimum, see e.g.~\cite[Thm.\ VII.4(ii)]{bertoin}. Taking $t\uparrow s$ we get $\e e^{-s\sigma}=\psi'(0)\Phi'(s)$ confirming the result of~\cite{last_passage}.

Finally, another result by Sparre Andersen~\cite{sparre_andersen_cyclical}, see also~\cite[Thm.\ XII.8.3]{feller},
states that the time of the maximum of a random walk `conditioned' to hit 0 at its terminal time, cf.\ Brownian bridge, has a uniform distribution.
This is a simple consequence of cyclical rearrangements of increments.
In our setting this result implies that our 6 random times have a uniform distribution on $[0,\sigma]$, i.e.
\begin{equation}\p(\overrightarrow F\in\D x|\sigma=t)=\frac{1}{t}\1{x\in[0,t]}\D x,\quad t>0.
\end{equation}
This result complements well-known uniform laws for L\'evy bridges~\cite{bridges,knight} stemming from the same result of Sparre Andersen, see also~\cite{alili,marchal} for an extension of the cyclical rearrangement idea.
In general, we have to assume that $X$ is a L\'evy process with distinct extrema
conditioned on $\{X_\sigma=0\}$.

\section{Random walk}\label{sec:rw}
Consider a random walk $S_i=\sum_{j=1}^i\zeta_j$ for $i=0,\ldots,n$, where $\zeta_1,\ldots, \zeta_n$ be iid random variables.
Let us condition this random walk on the (positive probability) event $\{S_n\in B\}$ for some Borel set~$B$; later we will take $B=\R$ and $B=[0,\infty)$.
Let
\[\sigma=\max\{i\leq n:S_i\leq 0\}\]
be the last time of $S_i$ in the non-positive half line.
Let
\begin{align*}
&\underline S=\min\{S_i:i\leq \sigma\}, &\overline S=\max\{S_i:i\leq \sigma\}
\end{align*}
and define the following 8 quantities:
\begin{align*}&N^-=\sum_{i=1}^\sigma\1{S_i\leq 0}, &N^+=\sum_{i=1}^\sigma\1{S_i\geq 0},\\
&\tilde N^-=\sum_{i=0}^{\sigma-1}\1{S_i\leq S_\sigma}, &\tilde N^+=\sum_{i=0}^{\sigma-1}\1{S_i\geq S_\sigma},\\
&\overrightarrow F=\max\{i\leq \sigma:S_i=\underline S\}, &\overrightarrow G=\max\{i\leq \sigma:S_i=\overline S\},\\
&\overleftarrow F=\sigma-\min\{i\leq \sigma:S_i=\underline S\}, &\overleftarrow G=\sigma-\min\{i\leq \sigma:S_i=\overline S\},
\end{align*}
see Fidure~\ref{fig_rw}.
\begin{figure}[h]
  \centering
  \includegraphics{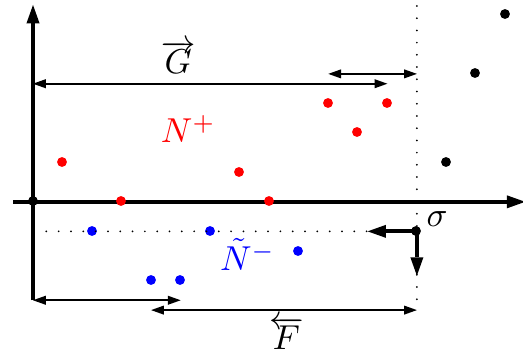}
  \caption{A realization of a random walk $S$ and the corresponding random times.}
\label{fig_rw}
\end{figure}
Moreover, we define a process $\hat S_i, i=0,\ldots,\sigma$ by
\[\hat S_i=S_\sigma-S_{\sigma-i},\]
which is just $-S$ time reversed at~$\sigma$.
\begin{prop}\label{prop:rw}
For a random walk $S_i,i=0,\ldots,n$ conditioned on $\{S_n\in B\}$ it holds that
\begin{itemize}
\item $\hat S$ has the law of $S$ considered up to~$\sigma$,
\item $N^-,\tilde N^+,\overrightarrow F,\overleftarrow G$ have the same distribution,
\item $N^+,\tilde N^-,\overleftarrow F,\overrightarrow G$ have the same distribution.
\end{itemize}
\end{prop}
\begin{proof}
For fixed $k=0,\ldots, n$ consider an event $\{\sigma=k\}=\{S_k\leq 0,S_i>0\text{ for all }k<i\leq n\}$ (assuming it has a positive probability).
Note that on the event $\{\sigma=k,S_n\in B\}$ the sequences $\zeta_1,\ldots, \zeta_k$ and $\zeta_k,\ldots, \zeta_1$ have the same law, and so
$S_i,i=0,\ldots,k$ and $\hat S_i,i=0,\ldots,k$ have the same laws.
Now the first statement follows by conditioning on~$\sigma$.

From the law equivalence of $S$ and $\hat S$ we get that
\[N^-\eqd \tilde N^+,\quad N^+\eqd \tilde N^-,\quad \overrightarrow F\eqd \overleftarrow G,\quad \overleftarrow F\eqd\overrightarrow G,\]
which is easily understood by drawing a picture. Otherwise, observe that $S_i\geq S_\sigma$ is the same as $0\geq S_\sigma-S_i=\hat S_{\sigma-i}$
and so
\begin{equation}\label{eq:N}
\tilde N^+=\sum_{i=0}^{\sigma-1} \1{\hat S_{\sigma-i}\leq 0}=\sum_{i=1}^{\sigma} \1{\hat S_i\leq 0}.
\end{equation}
This proves the first equality,
and the second follows similarly. Also
\begin{align}\label{eq:G}
\overleftarrow G&=\max\{\sigma-i\in[0,\sigma]:S_i=\overline S\}=\max\{j\in[0,\sigma]:S_\sigma-S_{\sigma-j}=S_{\sigma}-\overline S\}\nonumber\\
&=\max\{j\in[0,\sigma]:\hat S_j=\min\{\hat S_i:i\leq \sigma\}\}.
\end{align}
This proves the third statement and the fourth follows similarly.

Next, note that $\zeta_1,\ldots,\zeta_n$ conditioned on the event $\{S_n\in B\}$ are exchangeable random variables.
Thus it follows from Sparre-Andresen identity, see~\cite[Thm.\ XII.8.2]{feller}, that
$N^-$ and $\overrightarrow F$ have the same distribution (note that in their definitions $\sigma$ can be replaced by $n$).

Recall that $\zeta_1,\ldots,\zeta_k$ conditioned on the event $\{\sigma=k,S_n\in B\}$ are exchangeable random variables.
Thus on this event $N^+$ and $\overrightarrow G$ have the same distribution, which by conditioning is also true on the event $\{S_n\in B\}$.
\end{proof}
\begin{cor}\label{cor:rw}Assume that $\p(S_\sigma=0)>0$, then on the event $\{S_\sigma=0\}$ it holds that $N^-,N^+,\overrightarrow F,\overleftarrow F,\overrightarrow G,\overleftarrow G$
have the same distribution.
\end{cor}
\begin{proof}
The above proof requires only a small modification: we need to condition on $\{\sigma=k\}$ in the proof of $N^-\eqd\overrightarrow F$.
Finally, it follows that $N^-=\tilde N^-$ and $N^+=\tilde N^+$, showing that there is a single distribution.
\end{proof}

\section{\levy process}\label{sec:levy}
Extension of Proposition~\ref{prop:rw} to the case of a \levy process is standard, and hence only a sketch of it is presented in this note.
Consider a general \levy process~$X$, and without loss of generality assume that~$T=1$. 
Define $\sigma=\sup\{t\in[0,T]:X_t\leq 0\}$ and the corresponding time reversed process by
\[\hat X_t=X_{\sigma-}-X_{(\sigma-t)-},\quad t\in[0,\sigma),\]
where $X_{t-}$ denotes the left limit of $X$ at~$t$.
Define the random times as in Section~\ref{sec:intro} and in addition put
\begin{align*}
&\tilde N^-=\int_0^{\sigma}\1{X_t\leq X_{\sigma-}}\D t, &\tilde N^+=\int_0^{\sigma}\1{X_t\geq X_{\sigma-}}\D t.
\end{align*}

Consider a sequence of random walks $S^{(n)}$, defined by $S^{(n)}_i=X_{i/n},i=0,\ldots,n$, and the corresponding sequence of continuous approximations $X^{(n)}$ of $X$,
where points $(i/n,X_{i/n}), i=0,\ldots,n$ are connected by line segments (the appropriate topology is $M_1$, see~\cite[Ch.\ 3.3]{whitt}).
This setup and law equivalence of $\hat S$ and $S$ readily show that $\hat X$ has the same law as $X_t,t\in[0,\sigma).$  

Finally, we need to show that $n^{-1}{N^{(n)}}^-$ (corresponding to $S^{(n)}$) converges to $N^-$ (corresponding to $X$) a.s., and the same for the other quantities.
The case of a compound Poisson process is rather obvious, but requires to use another definition of $\overleftarrow F$ and $\overleftarrow G$:
\begin{align*}
&\overleftarrow F=\sigma-\min\{t\in[0,\sigma):X_t=\underline X_T\}, &\overleftarrow G=\sigma-\min\{t\in[0,\sigma):X_t=\overline X_{\sigma-}\}.
\end{align*}
Now suppose that $X$ is not a compound Poisson process. 
Then $\int_0^\sigma\1{X_t=0}\D t=0$ a.s, see~\cite[Prop.\ I.15]{bertoin}, and then also
 $\int_0^\sigma\1{X_t=X_{\sigma-}}\D t=0$ a.s., because of the law equivalence of $X$ and $\hat X$. 
 In addition, local extrema of $X$ are all distinct, see~\cite[Prop.\ VI.4]{bertoin}.
 Now the convergence of the scaled times for random walks to their \levy counterparts is clear,
see also the proof of~\cite[Lem.\ VI.15]{bertoin} presenting extension of the Sparre-Andersen identity to the \levy process case.

In conclusion, $N^-,\tilde N^+,\overrightarrow F,\overleftarrow G$ have the same distribution, and the same is true for $N^+,\tilde N^-,\overleftarrow F,\overrightarrow G$.
Moreover, Proposition~\ref{prop_nojumps} follows from Proposition~\ref{prop_nojumps2}, and the latter follows immediately from the general result,
by noticing that $X_{\sigma-}=0$ and thus $N^-=\tilde N^-,N^+=\tilde N^+$.
Finally, under conditions of Proposition~\ref{prop_nojumps} the time-reversed process $\hat X_t=-X_{(\sigma-t)-},t\in[0,\sigma)$
has the law of $X$ considered up to~$\sigma$.

\section*{Acknowledgment}
This note started from an observation that $N^-\eqd\tau_{\hat I}$ for a spectrally-negative \levy process, which together with~\eqref{eq:ack} asked for further investigations.
I would like to thank Hansjörg Albrecher for discussions which led to this observation, and Victor Rivero for drawing my attention to~\cite{alili,marchal}.
The support by Swiss National Science Foundation Project 200020\_143889 is gratefully acknowledged.

\end{document}